\documentclass{amsart}
\usepackage{amsmath}
\usepackage{amssymb}
\usepackage{amsfonts}

\newtheorem{theorem}{Theorem}
\theoremstyle{plain}

\newtheorem{corollary}{Corollary}

\newtheorem{definition}{Definition}

\newtheorem{lemma}{Lemma}

\newtheorem{proposition}{Proposition}

\numberwithin{equation}{section}
\input{tcilatex}

\begin{document}
\title[On some inequalities of Ostrowski type]{Ostrowski type inequalities
for $s-$logarithmically convex functions in the second sense with
applications}
\author{Mevl\"{u}t TUN\c{C}}
\address{Kilis 7 Aral\i k University, Faculty of Science and Arts,
Department of Mathematics, Kilis, 79000, Turkey.}
\email{mevluttunc@kilis.edu.tr}
\author{Ahmet Ocak Akdemir}
\address{A\u{g}r\i\ \.{I}brahim \c{C}e\c{c}en University, Faculty of Science
and Arts, Department of Mathematics, A\u{g}r\i , Turkey.}
\email{ahmetakdemir@agri.edu.tr}

\begin{abstract}
In this paper, we establish some new Ostrowski type inequalities for $s-$%
logarithmically convex functions. Some applications of our results to
P.D.F.'s \ and in numerical integration are given.
\end{abstract}

\maketitle

\section{INTRODUCTION}

Let $f:I\subset \left[ 0,\infty \right] \rightarrow 
\mathbb{R}
$ be a differentiable mapping on $I^{\circ }$, the interior of the interval $%
I$, such that $f^{\prime }\in L\left[ a,b\right] $ where $a,b\in I$ with $%
a<b $. If $\left\vert f^{\prime }\left( x\right) \right\vert \leq M$, then
the following inequality holds (see \cite{10}):

\begin{equation}
\left\vert f(x)-\frac{1}{b-a}\int_{a}^{b}f(u)du\right\vert \leq \frac{M}{b-a}%
\left[ \frac{\left( x-a\right) ^{2}+\left( b-x\right) ^{2}}{2}\right] .
\label{h.1.1}
\end{equation}

This inequality is well known in the literature as the Ostrowski inequality%
\textit{.}\textbf{\ }For some results which generalize, improve and extend
the inequality (\ref{h.1.1}) see (\cite{10}-\cite{a2}) and the references
therein.

Let us recall some known definitions and results which we will use in this
paper. The function $f:I\subset 
\mathbb{R}
\rightarrow \left[ 0,\infty \right) $ is said to be $\log -$convex or
multiplicatively convex if $\log t$ is convex, or, equivalenly, if for all $%
x,y\in I$ and $t\in \left[ 0,1\right] ,$ one has the inequality (See \cite{P}%
, p.7):%
\begin{equation*}
f\left( tx+\left( 1-t\right) y\right) \leq \left[ f\left( x\right) \right]
^{t}+\left[ f\left( y\right) \right] ^{\left( 1-t\right) }.
\end{equation*}

In \cite{at}, Akdemir and Tun\c{c} were introduced the class of $s-$%
logarithmically convex functions in the first sense as the following:

\begin{definition}
\label{mm} \textit{A function }$f:I\subset 
\mathbb{R}
_{0}\rightarrow 
\mathbb{R}
_{+}$\textit{\ is said to be }$s-$logarithmically convex in the first sense%
\textit{\ if \ \ \ \ \ \ \ \ \ \ \ \ }%
\begin{equation}
f\left( \alpha x+\beta y\right) \leq \left[ f\left( x\right) \right]
^{\alpha ^{s}}\left[ f\left( y\right) \right] ^{\beta ^{s}}  \label{m1}
\end{equation}%
for some $s\in \left( 0,1\right] $, where $x,y\in I$\textit{\ and }$\alpha
^{s}+\beta ^{s}=1.$
\end{definition}

In \cite{F}, authors introduced the class of $s-$logarithmically convex
functions in the second sense as the following:

\begin{definition}
\label{mmm} \textit{A function }$f:I\subset 
\mathbb{R}
_{0}\rightarrow 
\mathbb{R}
_{+}$\textit{\ is said to be }$s-$logarithmically convex in the second sense%
\textit{\ if \ \ \ \ \ \ \ \ \ \ \ \ }%
\begin{equation}
f\left( tx+\left( 1-t\right) y\right) \leq \left[ f\left( x\right) \right]
^{t^{s}}\left[ f\left( y\right) \right] ^{\left( 1-t\right) ^{s}}  \label{m2}
\end{equation}%
for some $s\in \left( 0,1\right] $, where $x,y\in I$\textit{\ and }$t\in %
\left[ 0,1\right] $\textit{.}
\end{definition}

Clearly, when taking $s=1$ in Definition \ref{m1} or Definition \ref{m2},
then $f$ becomes the standard logarithmically convex function on $I$.

The main purpose of this paper is to establish some new Ostrowski's type
inequalities for $s-$ logarithmically convex functions. We also give some
applications to P.D.F.'s and to midpoint formula.

\section{THE\ NEW\ RESULTS}

In order to prove our main results, we will use following Lemma which was
used by Alomari and Darus\textit{\ }(see \cite{11}):

\begin{lemma}
Let $f:I\subseteq 
\mathbb{R}
\rightarrow 
\mathbb{R}
$, be a differentiable mapping on $I$ where $a,b\in I$, with $a<b$. Let $%
f^{\prime }\in L[a,b],$ then the following equality holds;%
\begin{equation*}
f\left( x\right) -\frac{1}{b-a}\dint\limits_{a}^{b}f(u)du=\left( b-a\right)
\dint\limits_{0}^{1}p\left( t\right) f^{\prime }\left( ta+\left( 1-t\right)
b\right) dt
\end{equation*}%
for each $t\in \left[ 0,1\right] ,$ where 
\begin{equation*}
p\left( t\right) =\left\{ 
\begin{array}{cc}
t, & t\in \left[ 0,\frac{b-x}{b-a}\right] \\ 
&  \\ 
t-1, & t\in \left( \frac{b-x}{b-a},1\right]%
\end{array}%
\right.
\end{equation*}%
for all $x\in \left[ a,b\right] .$
\end{lemma}

\begin{theorem}
Let $I\supset \left[ 0,\infty \right) $ be an open interval and $%
f:I\rightarrow \left( 0,\infty \right) $ is differentiable mapping on $I$.
If $f^{\prime }\in L\left[ a,b\right] $ and $\left\vert f^{\prime
}\right\vert $ is $s-$logarithmically convex functions in the first sense on 
$\left[ a,b\right] ,$ $a,b\in I$, with $a<b,$ for some fixed $s\in \left( 0,1%
\right] ,$ then the following inequality holds:%
\begin{equation}
\left\vert f\left( x\right) -\frac{1}{b-a}\dint\limits_{a}^{b}f(u)du\right%
\vert \leq \left\vert f^{\prime }\left( b\right) \right\vert \Psi \left(
\tau ,s,a,b\right)  \label{1}
\end{equation}%
where%
\begin{eqnarray*}
&&\Psi \left( \tau ,s,a,b\right) \\
&=&\left\{ 
\begin{array}{cc}
\left( b-a\right) \left[ \frac{\tau ^{s\left( \frac{b-x}{b-a}\right) }\left(
2s\left( \frac{b-x}{b-a}\right) -1\right) +1}{2s^{2}\ln \tau }+\frac{\tau
^{s}-\tau ^{s\left( \frac{b-x}{b-a}\right) }\left( 2s\left( \frac{x-a}{b-a}%
\right) -1\right) }{2s^{2}\ln \tau }\right] & ,\tau <1 \\ 
&  \\ 
\frac{\left( a-x\right) ^{2}+\left( b-x\right) ^{2}}{2\left( b-a\right) } & 
,\tau =1%
\end{array}%
\right.
\end{eqnarray*}%
and%
\begin{equation*}
\tau =\frac{\left\vert f^{\prime }\left( a\right) \right\vert }{\left\vert
f^{\prime }\left( b\right) \right\vert }.
\end{equation*}
\end{theorem}

\begin{proof}
By Lemma 1 and since $\left\vert f^{\prime }\right\vert $ is $s-$%
logarithmically convex function in the first sense on $\left[ a,b\right] ,$
we have%
\begin{eqnarray*}
&&\left\vert f\left( x\right) -\frac{1}{b-a}\dint\limits_{a}^{b}f(u)du\right%
\vert \\
&\leq &\left( b-a\right) \left[ \dint\limits_{0}^{\frac{b-x}{b-a}%
}t\left\vert f^{\prime }\left( ta+\left( 1-t\right) b\right) \right\vert
dt+\dint\limits_{\frac{b-x}{b-a}}^{1}\left( 1-t\right) \left\vert f^{\prime
}\left( ta+\left( 1-t\right) b\right) \right\vert dt\right] \\
&\leq &\left( b-a\right) \left[ \dint\limits_{0}^{\frac{b-x}{b-a}%
}t\left\vert f^{\prime }\left( a\right) \right\vert ^{t^{s}}\left\vert
f^{\prime }\left( b\right) \right\vert ^{1-t^{s}}dt+\dint\limits_{\frac{b-x}{%
b-a}}^{1}\left( 1-t\right) \left\vert f^{\prime }\left( a\right) \right\vert
^{t^{s}}\left\vert f^{\prime }\left( b\right) \right\vert ^{1-t^{s}}dt\right]
.
\end{eqnarray*}%
If $\frac{\left\vert f^{\prime }\left( a\right) \right\vert }{\left\vert
f^{\prime }\left( b\right) \right\vert }=1,$ it easy to see that 
\begin{equation*}
\left\vert f\left( x\right) -\frac{1}{b-a}\dint\limits_{a}^{b}f(u)du\right%
\vert \leq \left\vert f^{\prime }\left( b\right) \right\vert \left[ \frac{%
\left( a-x\right) ^{2}+\left( b-x\right) ^{2}}{2\left( b-a\right) }\right] .
\end{equation*}%
If $\frac{\left\vert f^{\prime }\left( a\right) \right\vert }{\left\vert
f^{\prime }\left( b\right) \right\vert }<1,$ then $\left( \frac{\left\vert
f^{\prime }\left( a\right) \right\vert }{\left\vert f^{\prime }\left(
b\right) \right\vert }\right) ^{t^{s}}\leq \left( \frac{\left\vert f^{\prime
}\left( a\right) \right\vert }{\left\vert f^{\prime }\left( b\right)
\right\vert }\right) ^{st},$ we have%
\begin{eqnarray*}
&&\left\vert f\left( x\right) -\frac{1}{b-a}\dint\limits_{a}^{b}f(u)du\right%
\vert \\
&\leq &\left( b-a\right) \left\vert f^{\prime }\left( b\right) \right\vert 
\left[ \frac{\left( \frac{\left\vert f^{\prime }\left( a\right) \right\vert 
}{\left\vert f^{\prime }\left( b\right) \right\vert }\right) ^{s\left( \frac{%
b-x}{b-a}\right) }\left( 2s\left( \frac{b-x}{b-a}\right) -1\right) +1}{%
2s^{2}\ln \left( \frac{\left\vert f^{\prime }\left( a\right) \right\vert }{%
\left\vert f^{\prime }\left( b\right) \right\vert }\right) }\right. \\
&&\left. +\frac{\left( \frac{\left\vert f^{\prime }\left( a\right)
\right\vert }{\left\vert f^{\prime }\left( b\right) \right\vert }\right)
^{s}-\left( \frac{\left\vert f^{\prime }\left( a\right) \right\vert }{%
\left\vert f^{\prime }\left( b\right) \right\vert }\right) ^{s\left( \frac{%
b-x}{b-a}\right) }\left( 2s\left( \frac{x-a}{b-a}\right) -1\right) }{%
2s^{2}\ln \left( \frac{\left\vert f^{\prime }\left( a\right) \right\vert }{%
\left\vert f^{\prime }\left( b\right) \right\vert }\right) }\right] .
\end{eqnarray*}%
This completes the proof.
\end{proof}

\begin{corollary}
If we choose $\left\vert f^{\prime }\left( x\right) \right\vert \leq M$ in (%
\ref{1}), we obtain the inequality:%
\begin{equation*}
\left\vert f\left( x\right) -\frac{1}{b-a}\dint\limits_{a}^{b}f(u)du\right%
\vert \leq M\Psi \left( \tau ,s,a,b\right)
\end{equation*}%
where%
\begin{eqnarray*}
&&\Psi _{1}\left( \tau ,s,a,b\right) \\
&=&\left\{ 
\begin{array}{cc}
\left( b-a\right) \left[ \frac{M^{s\left( \frac{b-x}{b-a}\right) }\left(
2s\left( \frac{b-x}{b-a}\right) -1\right) +1}{2s^{2}\ln M}+\frac{%
M^{s}-M^{s\left( \frac{b-x}{b-a}\right) }\left( 2s\left( \frac{x-a}{b-a}%
\right) -1\right) }{2s^{2}\ln M}\right] & ,M<1 \\ 
&  \\ 
\frac{\left( a-x\right) ^{2}+\left( b-x\right) ^{2}}{2\left( b-a\right) } & 
,M=1%
\end{array}%
\right. .
\end{eqnarray*}
\end{corollary}

\begin{corollary}
If we choose $x=\frac{a+b}{2}$ in (\ref{1}), we obtain the inequality:%
\begin{equation*}
\left\vert f\left( \frac{a+b}{2}\right) -\frac{1}{b-a}\dint%
\limits_{a}^{b}f(u)du\right\vert \leq \left\vert f^{\prime }\left( b\right)
\right\vert \Psi \left( \tau ,s,a,b\right)
\end{equation*}%
where%
\begin{equation*}
\Psi \left( \tau ,s,a,b\right) =\left\{ 
\begin{array}{cc}
\left( b-a\right) \left[ \frac{\tau ^{\left( \frac{s\left( b-a\right) }{2}%
\right) }\left( s\left( b-a\right) -1\right) +1}{2s^{2}\ln \tau }+\frac{\tau
^{s}-\tau ^{\frac{s\left( b-a\right) }{2}}\left( s\left( b-a\right)
-1\right) }{2s^{2}\ln \tau }\right] & ,\tau <1 \\ 
&  \\ 
\frac{b-a}{4} & ,\tau =1%
\end{array}%
\right.
\end{equation*}%
and%
\begin{equation*}
\tau =\frac{\left\vert f^{\prime }\left( a\right) \right\vert }{\left\vert
f^{\prime }\left( b\right) \right\vert }.
\end{equation*}
\end{corollary}

\begin{corollary}
If we choose $s=1$ in (\ref{1}), we obtain the inequality:%
\begin{equation*}
\left\vert f\left( x\right) -\frac{1}{b-a}\dint\limits_{a}^{b}f(u)du\right%
\vert \leq \left\vert f^{\prime }\left( b\right) \right\vert \Psi \left(
\tau ,1,a,b\right)
\end{equation*}%
where%
\begin{equation*}
\Psi \left( \tau ,1,a,b\right) =\left\{ 
\begin{array}{cc}
\left( b-a\right) \left[ \frac{\tau ^{\left( \frac{b-x}{b-a}\right) }\left(
2\left( \frac{b-x}{b-a}\right) -1\right) +1}{2\ln \tau }+\frac{\tau -\tau
^{\left( \frac{b-x}{b-a}\right) }\left( 2\left( \frac{x-a}{b-a}\right)
-1\right) }{2\ln \tau }\right] & ,\tau <1 \\ 
&  \\ 
\frac{\left( a-x\right) ^{2}+\left( b-x\right) ^{2}}{2\left( b-a\right) } & 
,\tau =1%
\end{array}%
\right.
\end{equation*}%
and%
\begin{equation*}
\tau =\frac{\left\vert f^{\prime }\left( a\right) \right\vert }{\left\vert
f^{\prime }\left( b\right) \right\vert }.
\end{equation*}
\end{corollary}

\begin{theorem}
Let $I\supset \left[ 0,\infty \right) $ be an open interval and $%
f:I\rightarrow \left( 0,\infty \right) $ is differentiable mapping on $I$.
If $f^{\prime }\in L\left[ a,b\right] $ and $\left\vert f^{\prime
}\right\vert ^{q}$ is $s-$logarithmically convex functions in the first
sense on $\left[ a,b\right] ,$ $a,b\in I$, with $a<b,$ for some fixed $s\in
\left( 0,1\right] ,$ then the following inequality holds:%
\begin{equation}
\left\vert f\left( x\right) -\frac{1}{b-a}\dint\limits_{a}^{b}f(u)du\right%
\vert \leq \frac{\left\vert f^{\prime }\left( b\right) \right\vert }{\left(
p+1\right) ^{\frac{1}{p}}}\Psi \left( \tau ,s,a,b\right)  \label{2}
\end{equation}%
where%
\begin{eqnarray*}
&&\Psi \left( \tau ,s,a,b\right) \\
&=&\left\{ 
\begin{array}{cc}
\frac{1}{\left( b-a\right) ^{\frac{1}{p}}}\left[ \left( b-x\right) ^{\frac{%
p+1}{p}}\left( \frac{\tau ^{sq\left( \frac{b-x}{b-a}\right) }-1}{sq\ln \tau }%
\right) ^{\frac{1}{q}}+\left( x-a\right) ^{\frac{p+1}{p}}\left( \frac{\tau
^{sq}-\tau ^{sq\left( \frac{b-x}{b-a}\right) }}{sq\ln \tau }\right) ^{\frac{1%
}{q}}\right] & ,\tau <1 \\ 
&  \\ 
\left[ \frac{\left( b-x\right) ^{2}+\left( x-a\right) ^{2}}{b-a}\right] & 
,\tau =1%
\end{array}%
\right.
\end{eqnarray*}%
and%
\begin{equation*}
\tau =\frac{\left\vert f^{\prime }\left( a\right) \right\vert }{\left\vert
f^{\prime }\left( b\right) \right\vert }.
\end{equation*}%
for $q>1$ and $p^{-1}+q^{-1}=1.$
\end{theorem}

\begin{proof}
From Lemma 1 and by using the H\"{o}lder integral inequality, we have%
\begin{eqnarray*}
&&\left\vert f\left( x\right) -\frac{1}{b-a}\dint\limits_{a}^{b}f(u)du\right%
\vert \\
&\leq &\left( b-a\right) \left[ \left( \dint\limits_{0}^{\frac{b-x}{b-a}%
}t^{p}dt\right) ^{\frac{1}{p}}\left( \dint\limits_{0}^{\frac{b-x}{b-a}%
}\left\vert f^{\prime }\left( ta+\left( 1-t\right) b\right) \right\vert
^{q}dt\right) ^{\frac{1}{q}}\right. \\
&&\left. +\left( \dint\limits_{\frac{b-x}{b-a}}^{1}\left( 1-t\right)
^{p}dt\right) ^{\frac{1}{p}}\left( \dint\limits_{\frac{b-x}{b-a}%
}^{1}\left\vert f^{\prime }\left( ta+\left( 1-t\right) b\right) \right\vert
^{q}dt\right) ^{\frac{1}{q}}\right] .
\end{eqnarray*}%
Since $\left\vert f^{\prime }\right\vert $ is $s-$logarithmically convex
function in the first sense, we can write%
\begin{eqnarray*}
&&\left\vert f\left( x\right) -\frac{1}{b-a}\dint\limits_{a}^{b}f(u)du\right%
\vert \\
&\leq &\left( b-a\right) \left\vert f^{\prime }\left( b\right) \right\vert 
\left[ \left( \dint\limits_{0}^{\frac{b-x}{b-a}}t^{p}dt\right) ^{\frac{1}{p}%
}\left( \dint\limits_{0}^{\frac{b-x}{b-a}}\left( \frac{\left\vert f^{\prime
}\left( a\right) \right\vert }{\left\vert f^{\prime }\left( b\right)
\right\vert }\right) ^{qt^{s}}dt\right) ^{\frac{1}{q}}\right. \\
&&\left. +\left( \dint\limits_{\frac{b-x}{b-a}}^{1}\left( 1-t\right)
^{p}dt\right) ^{\frac{1}{p}}\left( \dint\limits_{\frac{b-x}{b-a}}^{1}\left( 
\frac{\left\vert f^{\prime }\left( a\right) \right\vert }{\left\vert
f^{\prime }\left( b\right) \right\vert }\right) ^{qt^{s}}dt\right) ^{\frac{1%
}{q}}\right] .
\end{eqnarray*}%
If $\frac{\left\vert f^{\prime }\left( a\right) \right\vert }{\left\vert
f^{\prime }\left( b\right) \right\vert }=1,$ then we have%
\begin{equation*}
\left\vert f\left( x\right) -\frac{1}{b-a}\dint\limits_{a}^{b}f(u)du\right%
\vert \leq \frac{\left\vert f^{\prime }\left( b\right) \right\vert }{\left(
p+1\right) ^{\frac{1}{p}}}\left[ \frac{\left( b-x\right) ^{2}+\left(
x-a\right) ^{2}}{b-a}\right] .
\end{equation*}%
On the other hand, if $\frac{\left\vert f^{\prime }\left( a\right)
\right\vert }{\left\vert f^{\prime }\left( b\right) \right\vert }<1,$ then $%
\left( \frac{\left\vert f^{\prime }\left( a\right) \right\vert }{\left\vert
f^{\prime }\left( b\right) \right\vert }\right) ^{qt^{s}}\leq \left( \frac{%
\left\vert f^{\prime }\left( a\right) \right\vert }{\left\vert f^{\prime
}\left( b\right) \right\vert }\right) ^{sqt},$ thereby%
\begin{eqnarray*}
&&\left\vert f\left( x\right) -\frac{1}{b-a}\dint\limits_{a}^{b}f(u)du\right%
\vert \\
&\leq &\left( b-a\right) \left\vert f^{\prime }\left( b\right) \right\vert 
\left[ \left( \dint\limits_{0}^{\frac{b-x}{b-a}}t^{p}dt\right) ^{\frac{1}{p}%
}\left( \dint\limits_{0}^{\frac{b-x}{b-a}}\left( \frac{\left\vert f^{\prime
}\left( a\right) \right\vert }{\left\vert f^{\prime }\left( b\right)
\right\vert }\right) ^{sqt}dt\right) ^{\frac{1}{q}}\right. \\
&&\left. +\left( \dint\limits_{\frac{b-x}{b-a}}^{1}\left( 1-t\right)
^{p}dt\right) ^{\frac{1}{p}}\left( \dint\limits_{\frac{b-x}{b-a}}^{1}\left( 
\frac{\left\vert f^{\prime }\left( a\right) \right\vert }{\left\vert
f^{\prime }\left( b\right) \right\vert }\right) ^{sqt}dt\right) ^{\frac{1}{q}%
}\right] .
\end{eqnarray*}%
By computing the above integrals the proof is completed.
\end{proof}

\begin{corollary}
If we choose $\left\vert f^{\prime }\left( x\right) \right\vert \leq M$ in (%
\ref{2}), we obtain the inequality:%
\begin{equation*}
\left\vert f\left( x\right) -\frac{1}{b-a}\dint\limits_{a}^{b}f(u)du\right%
\vert \leq \frac{M}{\left( p+1\right) ^{\frac{1}{p}}}\Psi \left( \tau
,s,a,b\right)
\end{equation*}%
where%
\begin{eqnarray*}
&&\Psi \left( \tau ,s,a,b\right) \\
&=&\left\{ 
\begin{array}{cc}
\frac{1}{\left( b-a\right) ^{\frac{1}{p}}}\left[ \left( b-x\right) ^{\frac{%
p+1}{p}}\left( \frac{M^{sq\left( \frac{b-x}{b-a}\right) }-1}{sq\ln M}\right)
^{\frac{1}{q}}+\left( x-a\right) ^{\frac{p+1}{p}}\left( \frac{%
M^{sq}-M^{sq\left( \frac{b-x}{b-a}\right) }}{sq\ln M}\right) ^{\frac{1}{q}}%
\right] & ,M<1 \\ 
&  \\ 
\left[ \frac{\left( b-x\right) ^{2}+\left( x-a\right) ^{2}}{b-a}\right] & 
,M=1%
\end{array}%
\right. .
\end{eqnarray*}
\end{corollary}

\begin{corollary}
If we choose $x=\frac{a+b}{2}$ in (\ref{2}), we obtain the inequality:%
\begin{equation*}
\left\vert f\left( \frac{a+b}{2}\right) -\frac{1}{b-a}\dint%
\limits_{a}^{b}f(u)du\right\vert \leq \left( \frac{b-a}{2}\right) \frac{%
\left\vert f^{\prime }\left( b\right) \right\vert }{\left( p+1\right) ^{%
\frac{1}{p}}}\Psi \left( \tau ,s,a,b\right)
\end{equation*}%
where%
\begin{equation*}
\Psi \left( \tau ,s,a,b\right) =\left\{ 
\begin{array}{cc}
\left( \frac{\tau ^{sq\left( \frac{b-a}{2}\right) }-1}{sq\ln \tau }\right) ^{%
\frac{1}{q}}+\left( \frac{\tau ^{sq}-\tau ^{sq\left( \frac{b-a}{2}\right) }}{%
sq\ln \tau }\right) ^{\frac{1}{q}} & ,\tau <1 \\ 
&  \\ 
\frac{1}{2} & ,\tau =1%
\end{array}%
\right.
\end{equation*}%
and%
\begin{equation*}
\tau =\frac{\left\vert f^{\prime }\left( a\right) \right\vert }{\left\vert
f^{\prime }\left( b\right) \right\vert }.
\end{equation*}
\end{corollary}

\section{APPLICATIONS FOR P.D.F's}

Let $X$ be a random variable taking values in the finite interval $[a,b],$
with the probability density function $f:[a,b]\rightarrow \lbrack 0,1]$ with
the cumulative distribution function $F(x)=\Pr (X\leq x)=\int_{a}^{b}f(t)dt.$

\begin{theorem}
Under the assumptions of Theorem 1, we have the inequality;%
\begin{equation*}
\left\vert \Pr (X\leq x)-\frac{1}{b-a}\left( b-E(x)\right) \right\vert \leq
\left\vert f^{\prime }\left( b\right) \right\vert \Psi \left( \tau
,s,a,b\right)
\end{equation*}%
where $E(x)$ is the expectation of $X$ and $\Psi \left( \tau ,s,a,b\right) $
as defined in Theorem 1.
\end{theorem}

\begin{proof}
The proof is immediate follows from the fact that; 
\begin{equation*}
E(x)=\int_{a}^{b}tdF(t)=b-\int_{a}^{b}F\left( t\right) dt.
\end{equation*}
\end{proof}

\begin{theorem}
Under the assumptions of Theorem 2, we have the inequality;%
\begin{equation*}
\left\vert \Pr (X\leq x)-\frac{1}{b-a}\left( b-E(x)\right) \right\vert \leq 
\frac{\left\vert f^{\prime }\left( b\right) \right\vert }{\left( p+1\right)
^{\frac{1}{p}}}\Psi \left( \tau ,s,a,b\right)
\end{equation*}%
where $E(x)$ is the expectation of $X$ and $\Psi \left( \tau ,s,a,b\right) $
as defined in Theorem 2.
\end{theorem}

\begin{proof}
Likewise the proof of the previous theorem, by using the fact that;%
\begin{equation*}
E(x)=\int_{a}^{b}tdF(t)=b-\int_{a}^{b}F\left( t\right) dt
\end{equation*}%
the proof is completed.
\end{proof}

\section{APPLICATIONS IN NUMERICAL INTEGRATION}

Let $d$ be a division of the closed interval $\left[ a,b\right] ,$ i.e., $%
d:a=x_{0}<x_{1}<...<x_{n-1}<x_{n}=b,$ and consider the midpoint formula%
\begin{equation*}
M\left( f,d\right) =\dsum\limits_{i=0}^{n-1}\left( x_{i+1}-x_{i}\right)
f\left( \frac{x_{i}+x_{i+1}}{2}\right) .
\end{equation*}%
It is well known tat if the mapping $f:\left[ a,b\right] \rightarrow 
\mathbb{R}
,$ is differentiable such that $f^{\prime \prime }\left( x\right) $ exists
on $\left( a,b\right) $ and $K=\sup_{x\in \left( a,b\right) }\left\vert
f^{\prime \prime }\left( x\right) \right\vert <\infty ,$ then 
\begin{equation*}
I=\dint\limits_{a}^{b}f\left( x\right) dx=M\left( f,d\right) +E_{M}\left(
f,d\right) ,
\end{equation*}%
where the approximation error $E_{M}\left( f,d\right) $ of the integral $I$
by the midpoint formula $M\left( f,d\right) $ satisfies 
\begin{equation*}
\left\vert E_{M}\left( f,d\right) \right\vert \leq \frac{K}{24}%
\dsum\limits_{i=0}^{n-1}\left( x_{i+1}-x_{i}\right) ^{3}.
\end{equation*}%
Now, it is time to give some new estimates for the remainder term $%
E_{M}\left( f,d\right) $ in terms of the first derivative of $s-$%
logaritmically convex functions.

\begin{proposition}
Let $I\supset \left[ 0,\infty \right) $ be an open interval and $%
f:I\rightarrow \left( 0,\infty \right) $ is differentiable mapping on $I$.
If $f^{\prime }\in L\left[ a,b\right] $ and $\left\vert f^{\prime
}\right\vert $ is $s-$logarithmically convex functions in the first sense on 
$\left[ a,b\right] ,$ $a,b\in I$, with $a<b,$ for some fixed $s\in \left( 0,1%
\right] ,$ then the following inequality holds for every division $d$ of $%
\left[ a,b\right] $:%
\begin{equation*}
\left\vert E_{M}\left( f,d\right) \right\vert \leq
\dsum\limits_{i=0}^{n-1}\left\vert f^{\prime }\left( x_{i+1}\right)
\right\vert \Psi \left( \tau ,s,x_{i},x_{i+1}\right)
\end{equation*}%
where%
\begin{eqnarray*}
&&\Psi \left( \tau ,s,x_{i},x_{i+1}\right) \\
&=&\left\{ 
\begin{array}{cc}
\left( x_{i+1}-x_{i}\right) \left[ \frac{\tau ^{\left( \frac{s\left(
x_{i+1}-x_{i}\right) }{2}\right) }\left( s\left( x_{i+1}-x_{i}\right)
-1\right) +1}{2s^{2}\ln \tau }+\frac{\tau ^{s}-\tau ^{\frac{s\left(
x_{i+1}-x_{i}\right) }{2}}\left( s\left( x_{i+1}-x_{i}\right) -1\right) }{%
2s^{2}\ln \tau }\right] & ,\tau <1 \\ 
&  \\ 
\frac{x_{i+1}-x_{i}}{4} & ,\tau =1%
\end{array}%
\right.
\end{eqnarray*}%
and%
\begin{equation*}
\tau =\frac{\left\vert f^{\prime }\left( x_{i}\right) \right\vert }{%
\left\vert f^{\prime }\left( x_{i+1}\right) \right\vert }.
\end{equation*}
\end{proposition}

\begin{proof}
By applying Corollary 2 on the subintervals $\left[ x_{i},x_{i+1}\right] ,$ $%
\left( i=0,1,...,n-1\right) $ of the division $d,$ we have%
\begin{equation*}
\left\vert f\left( \frac{x_{i}+x_{i+1}}{2}\right) -\frac{1}{x_{i+1}-x_{i}}%
\int_{x_{i}}^{x_{i+1}}f(u)du\right\vert \leq \left\vert f^{\prime }\left(
x_{i+1}\right) \right\vert \Psi \left( \tau ,s,x_{i},x_{i+1}\right) .
\end{equation*}%
By summing over $i$ from $0$ to $n-1$, it is easy to see that%
\begin{equation*}
\left\vert M\left( f,d\right) -\int_{a}^{b}f(x)dx\right\vert \leq
\sum_{i=0}^{n-1}\left\vert f^{\prime }\left( x_{i+1}\right) \right\vert \Psi
\left( \tau ,s,x_{i},x_{i+1}\right)
\end{equation*}%
which completes the proof.
\end{proof}

\begin{proposition}
Let $I\supset \left[ 0,\infty \right) $ be an open interval and $%
f:I\rightarrow \left( 0,\infty \right) $ is differentiable mapping on $I$.
If $f^{\prime }\in L\left[ a,b\right] $ and $\left\vert f^{\prime
}\right\vert ^{q}$ is $s-$logarithmically convex functions in the first
sense on $\left[ a,b\right] ,$ $a,b\in I$, with $a<b,$ for some fixed $s\in
\left( 0,1\right] ,$ then the following inequality holds for every $d$ of $%
\left[ a,b\right] $:%
\begin{equation*}
\left\vert E_{M}\left( f,d\right) \right\vert \leq \sum_{i=0}^{n-1}\left( 
\frac{x_{i+1}-x_{i}}{2}\right) \frac{\left\vert f^{\prime }\left(
x_{i+1}\right) \right\vert }{\left( p+1\right) ^{\frac{1}{p}}}\Psi \left(
\tau ,s,x_{i},x_{i+1}\right)
\end{equation*}%
where%
\begin{equation*}
\Psi \left( \tau ,s,x_{i},x_{i+1}\right) =\left\{ 
\begin{array}{cc}
\left( \frac{\tau ^{sq\left( \frac{x_{i+1}-x_{i}}{2}\right) }-1}{sq\ln \tau }%
\right) ^{\frac{1}{q}}+\left( \frac{\tau ^{sq}-\tau ^{sq\left( \frac{%
x_{i+1}-x_{i}}{2}\right) }}{sq\ln \tau }\right) ^{\frac{1}{q}} & ,\tau <1 \\ 
&  \\ 
\frac{1}{2} & ,\tau =1%
\end{array}%
\right.
\end{equation*}%
and%
\begin{equation*}
\tau =\frac{\left\vert f^{\prime }\left( x_{i}\right) \right\vert }{%
\left\vert f^{\prime }\left( x_{i+1}\right) \right\vert }.
\end{equation*}%
for $q>1$ and $p^{-1}+q^{-1}=1.$
\end{proposition}

\begin{proof}
The proof of the result is similar to the proof of the Proposition 1, by
applying Corollary 5.
\end{proof}


\begin{thebibliography}{9}
\bibitem{at} A.O. Akdemir and M. Tun\c{c}, \textit{On some integral
inequalities for }$s$\textit{-logarithmically convex functions}, Submitted.

\bibitem{10} A. Ostrowski,\textit{\ \"{U}ber die Absolutabweichung einer
differentierbaren Funktion von ihren Integralmittelwert}, Comment. Math.
Helv., 10, 226-227, (1938).

\bibitem{11} M. Alomari and M. Darus, \textit{Some Ostrowski type
inequalities for convex functions with applications}, RGMIA Res. Rep. Coll.,
(2010) 13, 2, Article 3. [ONLINE: http://ajmaa.org/RGMIA/v13n2.php].

\bibitem{a1} M.E. \"{O}zdemir, H. Kavurmac\i , E. Set, \textit{Ostrowski's
type inequalities for }$(\alpha ,m)-$\textit{convex functions}, KYUNGPOOK
Math. J. 50 (2010) 371--378.

\bibitem{2} H. Kavurmac\i , M.E. \"{O}zdemir and M. Avc\i , \textit{New
Ostrowski type inequalities for }$m-$\textit{convex functions and
applications}, Hacettepe Journal of Mathematics and Statistics, Volume 40
(2) (2011), 135 -- 145.

\bibitem{a2} E. Set, \textit{New inequalities of Ostrowski type for mappings
whose derivatives are }$s$\textit{-convex in the second sense via fractional
integrals}, Comput. Math. Appl., 63 (2012) 1147-1154.

\bibitem{P} J. Pecaric, F. Proschan and Y.L. Tong, \textit{Convex Functions,
Partial Orderings and Statistical Applications}, Academic Press, Inc., 1992.

\bibitem{F} B-Y. Xi and F. Qi, \textit{Some integral inequalities of
Hermite-Hadamard type for }$s-$\textit{logarithmically convex functions},
Acta Mathematica Scientis, English Series, 2013, Vol. 33, No. ?, pp. 1--12.
\end{thebibliography}
\end{document}